\documentclass[]{amsart}
\usepackage{amscd,amsthm,amssymb,amsfonts,amsmath,euscript}

\theoremstyle{plain}
\newtheorem{thm}{Theorem}[section]
\newtheorem{lemma}[thm]{Lemma}
\newtheorem{prop}[thm]{Proposition}
\newtheorem{cor}[thm]{Corollary}

\theoremstyle{definition}
\newtheorem{defn}[thm]{Definition}

\theoremstyle{remark}
\newtheorem{remark}[thm]{Remark}

\newcommand{\nc}{\newcommand}

\def\makeop#1{\expandafter\def\csname#1\endcsname
  {\mathop{\rm #1}\nolimits}\ignorespaces}
\makeop{Hom}   \makeop{End}   \makeop{Aut}   \makeop{Isom}  \makeop{Pic} 
\makeop{Gal}   \makeop{ord}   \makeop{Char}  \makeop{Div}   \makeop{Lie} 
\makeop{PGL}   \makeop{Corr}  \makeop{PSL}   \makeop{sgn}   \makeop{Spf}
\makeop{Spec}  \makeop{Tr}    \makeop{Nr}    \makeop{Fr}    \makeop{disc}
\makeop{Proj}  \makeop{supp}  \makeop{ker}   \makeop{im}    \makeop{dom}
\makeop{coker} \makeop{Stab}  \makeop{SO}    \makeop{SL}    \makeop{SL}
\makeop{Cl}    \makeop{cond}  \makeop{Br}    \makeop{inv}   \makeop{rank}
\makeop{id}    \makeop{Fil}   \makeop{Frac}  \makeop{GL}    \makeop{SU}
\makeop{Nrd}   \makeop{Sp}    \makeop{Tr}    \makeop{Trd}   \makeop{diag}
\makeop{Res}   \makeop{ind}   \makeop{depth} \makeop{Tr}    \makeop{st}
\makeop{Ad}    \makeop{Int}   \makeop{tr}    \makeop{Sym}   \makeop{can}
\makeop{length}\makeop{SO}    \makeop{torsion} \makeop{GSp} \makeop{Ker}
\makeop{Adm}   \makeop{Mat}
\def\makebb#1{\expandafter\def
  \csname bb#1\endcsname{{\mathbb{#1}}}\ignorespaces}
\def\makebf#1{\expandafter\def\csname bf#1\endcsname{{\bf
      #1}}\ignorespaces} 
\def\makegr#1{\expandafter\def
  \csname gr#1\endcsname{{\mathfrak{#1}}}\ignorespaces}
\def\makescr#1{\expandafter\def
  \csname scr#1\endcsname{{\EuScript{#1}}}\ignorespaces}
\def\makecal#1{\expandafter\def\csname cal#1\endcsname{{\mathcal
      #1}}\ignorespaces} 

\def\doLetters#1{#1A #1B #1C #1D #1E #1F #1G #1H #1I #1J #1K #1L #1M
                 #1N #1O #1P #1Q #1R #1S #1T #1U #1V #1W #1X #1Y #1Z}
\def\doletters#1{#1a #1b #1c #1d #1e #1f #1g #1h #1i #1j #1k #1l #1m
                 #1n #1o #1p #1q #1r #1s #1t #1u #1v #1w #1x #1y #1z}
\doLetters\makebb   \doLetters\makecal  \doLetters\makebf
\doLetters\makescr 
\doletters\makebf   \doLetters\makegr   \doletters\makegr
     \def\qed{\qedmark\medbreak}%
\def\qedmark{{\enspace\vrule height 6pt width 5pt depth 1.5pt}}%
    
\def\Gm{{{\bbG}_{\rm m}}}   

\normalsize

\makeop{Bl}

\def\Fq{{\bbF}_q}

\newcommand{\Z}{\mathbb Z}
\newcommand{\Q}{\mathbb Q}
\newcommand{\R}{\mathbb R}

\renewcommand{\H}{\mathbb H}  
\newcommand{\A}{\mathbb A}    

\newcommand{\pr}{\indent }

\newcommand{\isoto}{\stackrel{\sim}{\longrightarrow}}
\nc{\embed}{\hookrightarrow}

\nc{\ol}{\overline}
\nc{\wt}{\widetilde}
\nc{\opp}{\mathrm{opp}}

\def\wh{\widehat}

\makeop{Ram}
\makeop{Rep}

\begin{document}
\renewcommand{\thefootnote}{\fnsymbol{footnote}}
\setcounter{footnote}{-1}
\numberwithin{equation}{section}


\title{Variations of mass formulas for definite division algebras}
\author{Chia-Fu Yu}
\address{
Institute of Mathematics, Academia Sinica and NCTS (Taipei Office)\\
6th Floor, Astronomy Mathematics Building \\
No. 1, Roosevelt Rd. Sec. 4 \\ 
Taipei, Taiwan, 10617} 
\email{chiafu@math.sinica.edu.tw}
\address{
The Max-Planck-Institut f\"ur Mathematik \\
Vivatsgasse 7, Bonn \\
Germany 53113} 
\email{chiafu@mpim-bonn.mpg.de}


\date{\today}

\begin{abstract}
The aim of this paper is to organize some known mass formulas 
arising from a definite central division algebra over a global field
and to deduce some more new ones.  
\end{abstract} 

\maketitle


\def\Mass{{\rm Mass}}
\def\DS{{\rm DS}}
\def\vol{{\rm vol}}
\def\ad{{\rm ad}}
\def\pr{{\rm pr}}
\def\rad{{\rm rad}}

\section{Introduction}
\label{sec:01}

Let $K$ be a global field and $A$ be the ring of $S$-integers in $K$,
where $S$ be a non-empty finite set of places of $K$ that contains 
all Archimedean places if $K$ is a number field. 
Let $D$ be a central division algebra $D$ of degree $n\ge 2$
over $K$ 
that is is {\it definite} relative to $S$. 
This means that for all places $v\in S$ 
the completion $D_v:=D\otimes_K K_v$ of $D$ at 
$v$  remains a {\it division} algebra over $K_v$. 
When $K$ is a number field, the definite condition implies that 
$K$ is necessarily totally real and that $D$ is a totally definite 
{\it quaternion} algebra (the completions at all real places are Hamiltion
quaternion algebras). There are extensive studies for these
quaternion algebras over totally real fields in various aspects (mass
formulas, class number formulas,  modular forms, theta series etc) by
Eichler and many others.
In this paper we studies three mass formulas arising from the 
algebra $D$ and an $A$-order $R$ in $D$.
 

The first one is the more classical mass associated to the pair
$(D,R)$ using algebras, which 
dates back to Deuring and Eichler; see \cite{eichler:crelle55},
cf. \cite{vigneras}.
Let $\{I_1, \dots, I_h\}$ be a complete set of representatives of the
right locally principal ideal classes of $R$. Define the 
{\it mass of $(D,R)$} by
\begin{equation}
  \label{eq:11}
  \Mass(D,R):=\sum_{i=1}^h [R_i^\times:A^\times]^{-1},  
\end{equation}
where $R_i$ is the left order of $I_i$. See Section~\ref{sec:23} for
detailed discussions.

Another two masses are defined by group theory. Recall that if 
a reductive group
$G$ over $K$ has finite $S$-arithmetic subgroups, then for any open
compact subgroup $U\subset G(\A^S)$, where $\A^S$ is the
prime-to-$S$ adele ring of $K$, one can associate the mass $\Mass(G,U)$
as the weight sum over the double coset space $\DS(G,U)=G(K)\backslash
G(\A^S)/U$ (see Section~\ref{sec:22}).  
Now let $G$ be the
multiplicative group of $D$ viewed as an algebraic group over 
$K$. 
Let $G_1$ denote the reduced norm one subgroup
of $G$ and $G^\ad$ the adjoint group of $G$. 
The definite condition implies 
that the groups $G_1(K)$ and $G^\ad(K)$ have finite $S$-arithmetic
subgroups (Section~\ref{sec:22}).  
Put $U:=\wh R^\times\subset G(\A^S)$, 
where $\wh R=\prod_{v\not\in S} R_v$ is the profinite
completion of $R$. 
Put $U_1:=U\cap G_1(\A^S)$ and let $U^\ad\subset
G^\ad(\A^S)$ be the image of $U$. Using the vanishing of the first Galois 
cohomology one 
shows that the induced projection $\pr:G(\A^S)\to G^\ad (\A^S)$ is
open and surjective, particularly that $U^\ad$ is an open compact
subgroup. Therefore we have defined the masses $\Mass(G_1,U_1)$ 
and $\Mass(G^\ad,U^\ad)$.

The main contents of this article are to
compare these masses and to compute them explicitly. 
Our first main result is the following; see Theorem~\ref{32} and
Corollary~\ref{38}. 



\begin{thm}\label{11} 
We have
  \begin{equation}
    \label{eq:1.2}
     \Mass(D,R)=h_A\cdot \Mass(G^\ad, U^\ad), 
  \end{equation}
  where $h_A$ is the class number of $A$. Moreover, we have 
\begin{equation}    
\label{eq:1.3}
    \Mass(G^\ad, U^\ad)= c(S,U) \cdot \Mass(G_1,U_1), 
  \end{equation}
where 
\begin{equation}
  \label{eq:1.4}
  c(S,U)=
  \begin{cases}
   n^{-(|S|-1)} [\wh A^\times: \Nr(U)]  & \text{if $K$ is a
   function field;} \\ 
   2^{-(|S|-|\infty|-1)} [\wh A^\times: \Nr(U)]  & \text{if $K$ is a
   totally real field.} \\  
  \end{cases}
\end{equation}
Here $\Nr:G(\A^S)\to \A^{S,\times}$ denotes the reduced norm map, $\wh
A=\prod_v A_v$ is the profinite completion of $A$ and $\infty$ is the
set of Archimedean places of the number field $K$.  
\end{thm}

Thus, knowing one of the three masses will allow us to compute the
other two.  
For $\Mass(D,R)$ we obtain the following formula; see Theorem~\ref{42}. 


\begin{thm}\label{12}  
We have 
  \begin{equation}
    \label{eq:1.5}
    \Mass(D,R)=\frac{h_A}{n^{|S|-1}}\cdot 
    \prod_{i=1}^{n-1} |\zeta_K(-i) |\cdot \prod_{v} \lambda_v(R_v),
  \end{equation}
  where $\zeta_K(s)$ is the Dedekind zeta function of $K$, $v$
  runs through all non-Archimedean places of $K$ and the local term
  $\lambda_v(R_v)$ 
  is defined in (\ref{eq:4.10}). 
\end{thm}






%

In the case where $D$ is a quaternion algebra, i.e. $n=2$, 
Theorem~\ref{12} gives rise to a more explicit formula
(see Corollary~\ref{43}) which was obtained first by K\"orner 
in the number field case 
(see \cite[Theorem 1]{korner:1987}, also see \cite{korner:1985} for the
computation). The mass formula proved by K\"orner is used further
by Brzezinski \cite{brzezinski:classno} 
to classify orders in all definite quaternion algebras over 
$\Q$ with class number one. We remark that definite Eichler orders
$\calO$ of
class number $h(\calO)\le 2$ are classified in Kirschmer and Voight 
\cite{kirschmer-voight:quat_orderSIAM2010} 





The proof of (\ref{eq:1.2}) is analyzing the action of the Picard
group $\Pic(A)$ on the double coset space $\DS(G,U)$ and comparing the
two masses from the definition. The proof of (\ref{eq:1.3}) is first
to reduce the case where $R$ is maximal, and in this case we compute
the factor $c(S,U)$ from the explicit formula for $\Mass(G^\ad,
U^\ad)$ and $\Mass(G_1,U_1)$.  

The proof of Theorem~\ref{12} is similar to that in K\"orner
\cite{korner:1987} which starts the (known) 
mass formula for maximal orders and computes explicitly the local terms.  
Using the interpretation of masses as volumes of fundamental domains,
we can reduce the mass formula for maximal orders to the classical case 
(i.e. $S=\infty$ or ``the place at infinity'' in the function field
case), which is well known due to Eichler in the number field case 
and is due to Denert-Van Geel \cite{denert-geel:classno88} 
in the function field case (also see different proofs
in Wei and the author~\cite{wei-yu:mass_division}). 
In the latter case, the mass formula
was used in Denert and Van Geel~\cite{denert-geel:crelle1986} to prove 
the cancellation property for $\Fq[t]$-orders in definite 
central division algebras over $K=\Fq(t)$. 

Though there is no new idea added in the proof of Theorem~\ref{12}, it is
convenient to have an explicit formula 
for some arithmetic and geometric applications 
(e.g. estimating class numbers and computing certain supersingular
objects, see \cite{brzezinski:classno, kirschmer-voight:quat_orderSIAM2010, 
denert-geel:crelle1986, gekeler:quaternion, gekeler:mass, yu:thesis, 
yu:mass_hb, yu:ss_siegel, yu:gmf, yu-yu:mass_surface}).  
  
We remark that mass formulas for more general groups
have been determined by Prasad \cite{prasad:s-volume}, Gan
and Gross \cite{gross-gan}, Shimura (cf. \cite{shimura:1999}) and
Gan-Hanke-Yu \cite{gan-hanke-yu, gan-yu}. We refer the interested
reader to their papers for more mass formulas.

This paper is organized as follows. Section~\ref{sec:02} discusses 
variants of masses arising from a definite central division
algebra. 
Section~\ref{sec:03} compares these masses (Theorem~\ref{11}) and
deduces a mass formula (Theorem~\ref{12}) in the case 
where $R$ is a maximal order. In
Section~\ref{sec:04} we compute the local indices and prove
Theorem~\ref{12}. 
The last section discusses a mass formula for types of orders. 

 
\section{Definitions of masses}
\label{sec:02}

\subsection{Setting}
\label{sec:21}

Let $K$ be a global field.  Let $S$ be a non-empty finite set of places
of $K$ that contains all Archimedean places if $K$ is a number field
or contains a fixed place $\infty$ if $K$ is a global function field. 
We also write $\infty$ for the set of Archimedean places 
when $K$ is a number field.
Let $A$ be the ring of
$S$-integers. If $K$ is a number field and $S=\infty$, 
then $A$ is nothing but the ring of integers in $K$ 
which is usually denoted by $O_K$. 
Let $V^K$ (resp. $V^K_f$) denote the set of all
(resp. all non-Archimedean) places of $K$.
There is a natural one-to-one bijection
between the set of places $v\not\in S$ and the set ${\rm Max}(A)$ 
of non-zero prime ideals of $A$.  
For any place $v$ of $K$, let $K_v$ denote the
completion of $K$ at $v$. If $v$ is non-Archimedean, then let $O_v$ 
denote the valuation ring, $k(v)$ the residue field and $q_v$
its cardinality.  
In case $v\not\in S$, one also writes $A_v$ for $O_v$, 
the completion of $A$ at $v$. Write
$|I|:=|A/I|$ (resp. $|I_v|:=|A_v/I_v|$) if $I\subset A$
(resp. $I_v\subset A_v$) is a non-zero integral ideal. 
Let $\A$ denote the adele ring of $K$, $\A^S:=\prod'_{v\not\in S} K_v$ 
the prime-to-$S$ adele ring of $K$ and $\A_S:=\prod_{v\in S} K_v$. One
has $\A=\A_S\times \A^S$. Write $\wh A=\prod_{v\not \in S}A_v$ for the
profinite completion of $A$. For any finitely generated $A$-module
$R$, write $\wh R:=R\otimes_A \wh A$.

Let $G$ be a reductive algebraic group over $K$. Recall that an
{\it $S$-arithmetic subgroup of $G$} is a subgroup of the group 
$G(K)$ of $K$-rational
points which is commensurable to the intersection of $G(K)$ 
with an open compact subgroup $U$ of $G(\A^S)$. 
If an $S$-arithmetic subgroup of $G$ is finite, 
then every $S$-arithmetic subgroup of $G$ is finite. 

For any open compact subgroup $U\subset G(\A^S)$, we write
$\DS(G,U)$ for the double coset space 
$G(K)\backslash G(\A^S)/U$. 
By the finiteness of class numbers due to Harish-Chandra and Borel
\cite{borel:finite}, the set $\DS(G,U)$ is always finite. 



\subsection{Mass of $(G,U)$}
\label{sec:22}

Suppose that any $S$-arithmetic subgroup of $G$ is finite. For any
open compact subgroup $U\subset G(\A^S)$, we define the mass of
$(G,U)$ by 
\begin{equation}
  \label{eq:21}
  \Mass(G,U):=\sum_{i=1}^h |\Gamma_{c_i}|^{-1}, 
\end{equation}
where $c_1,\dots, c_h$ are representatives for the double coset space
$\DS(G,U)$ and $\Gamma_{c_i}:=G(K)\cap c_i U c_i^{-1}$ 
for $i=1,\dots, h$.  
Note that $\Gamma_{c_i}=\{g\in G(K)\mid g(c_iU)=c_iU\}$ and it is 
finite.

If $G_S:=G(\A_S)$ is compact, then any $S$-arithmetic subgroup is
discretely embedded into the compact group $G_S$ and hence is
finite. In this case the mass $\Mass(G,U)$ associated to 
$(G,U)$ is defined for any open compact subgroup $U\subset G(\A^S)$. 

There are examples of groups 
$G$ with finite $S$-arithmetic subgroups whose $S$-component 
$G_S$ needs not to be compact. For example, let
$D$ be a definite quaternion algebra over $\Q$ (with $S=\infty$) 
and $G:=D^\times$ be
the multiplicative group of $D$. 
Then the group $G(\R)=\H^\times$ of $\R$-points, which is 
the group of units in the Hamilton
quaternion algebra, is not compact. However, any arithmetic subgroup
of $G(\Q)$ is finite. Another example is the multiplicative 
group $G$ associated to a
definite central division algebra $D$ over a function field $K$ with $|S|=1$. 
 
Note that if $G_S:=G(\A_S)$ is compact, then the group $G(K)$ is
identified with a discrete subgroup in $G(\A^S)$ through the diagonal
embedding and the quotient
topological space $G(K)\backslash G(\A^S)$ is compact.
This space provides a fertile ground for studying harmonic
analysis. 
Slightly more general, 
one has the following equivalent statements which characterize
the groups with finite $S$-arithmetic subgroups:

\begin{prop} The following statements are equivalent.
  \begin{enumerate}
  \item Any $S$-arithmetic subgroup of $G(K)$ is finite.
  \item The group $G(K)$ is discretely embedded into the locally compact
    topological group $G(\A^S)$. 
  \item The group $G(K)$ is discretely embedded into the locally compact
    topological group $G(\A^S)$ and the quotient topological 
    space $G(K)\backslash
    G(\A^S)$ is compact. 
  \end{enumerate}  
\end{prop}
\begin{proof}
  See a proof in Gross \cite{gross:amf}. \qed
\end{proof}

In general, it is very difficult to calculate the class number
$|\DS(G,U)|$ explicitly. 
The mass $\Mass(G,U)$ associated to $(G,U)$, by its definition, is a
weighted class number. It is weighted according to the extra
symmetries of each double coset. The mass is easier to compute and it
provides a good lower bound for the class number. 
On the other hand, one can interpret  
$\Mass(G,U)$ as the volume of a fundamental domain.

\begin{lemma}\label{22}
  Let $G$ be a reductive group over $K$ with finite $S$-arithmetic
  subgroups. Then $\Mass(G,U)=\vol(G(K)\backslash G(\A^S))
  \vol(U)^{-1}$ for any Haar measure on $G(\A^S)$ and the counting
  measure  for the discrete subgroup $G(K)$. In particular if the Haar
  measure is chosen so that $\vol(U)=1$, then  
  $\Mass(G,U)=\vol(G(K)\backslash G(\A^S))$.
\end{lemma}
\begin{proof}
  Let $c_1,\dots, c_h$ be representatives for $\DS(G,U)$. One has 
\[ G(\A^S)=\coprod_{i=1}^h G(K)c_i U \]
and for each class
\[ \vol(G(K)\backslash G(K)c_i U)=\frac{\vol(U)}{\vol(G(K)\cap c_i
  U c_i^{-1})}=\vol(U) |\Gamma_{c_i}|^{-1}. \]
Then we get
\[ \vol(G(K)\backslash G(\A^S))=\sum_{i=1}^h \vol(G(K)\backslash
G(K)c_i U)=\vol(U)\cdot \Mass(G,U).\quad \text{\qed} \] 
\end{proof}
This interpretation of $\Mass(G,U)$ 
allows us to compare the masses $\Mass(G,U)$ and
$\Mass(G, U')$ for different open compact subgroups $U$ and $U'$ in
$G(\A^S)$. Indeed by Lemma~\ref{22} we have
\begin{equation}
  \label{eq:22}
  \Mass(G,U')=\Mass(G,U) [U:U'], 
\end{equation}
where the index $[U:U']$ is defined by 
\begin{equation}
  \label{eq:23}
  [U:U']:=[U:U''] [U':U'']^{-1}
\end{equation}
for any open compact subgroup $U''\subset U\cap U'$. 

\subsection{Mass of $(D,R)$}
\label{sec:23}
Let $D$ be a central central
algebra over $K$ which is definite relative to $S$. 
This means that the completion
$D_v$ at $v$, for any place $v\in S$, is a central
{\it division} algebra over $K_v$. In particular $D$ is a division
algebra.  
In the literature, definite central simple algebras are exactly 
those that do not satisfy the $S$-Eichler condition. 

Let $S_D\subset V^K$ denote the finite set of ramified places for
$D$. When $D$ is a quaternion algebra, the definite condition for $D$ 
simply means that $S\subset S_D$. However, the condition
$S\subset S_D$ is not sufficient to conclude that $D$ is definite in
general. One also needs to know the
invariants of $D$.

Let $R$ be an $A$-order in $D$.  Two right
$R$-ideals $I$ and $I'$ are said to be {\it equivalent}, which we
denote by $I_1\sim I_2$, if there is an element $g\in
D^\times$ such that $I'=gI$. In other words, $I_1\sim I_2$ if and only
if $I$ and $I'$ are isomorphic as right $R$-modules. Let $\Cl (R)$
denote the set of 
equivalence classes  of locally free right $R$-ideals.
It is well known that the set $\Cl
(R)$ is always finite, and that this set 
can be parametrized by an adelic class space:
\[ \Cl(R)\simeq D^\times \backslash D^\times_{\A^S}/\wh R^\times, \]
where $\wh R=\prod_{v\not \in S} R_v$ ($R_v=R\otimes_A A_v$) 
is the profinite completion of $R$ and
$D_{\A^S}=D\otimes_K \A^S$ is the attached prime-to-$S$ adele ring of
$D$. 

Let $I_1,\dots, I_h$ be representatives for the ideal classes in
$\Cl(R)$. Let $R_i$ be the left order of $I_i$. Then
$[R_i^\times:A^\times]$ is finite. This follows from the Dirichlet
theorem that $A^\times$ is finitely generated $\Z$-module of rank
$|S|-1$ and the following exact sequence:
\[ 1\to R_{i,1}^\times/A^\times_1 \to R_i^\times/A^\times \to
\Nr(R_i^\times)/\Nr(A^\times)\to 1,\]
where $\Nr:D^\times \to K^\times$ is the reduced norm,
$R_{i,1}^\times=R^\times_{i}\cap \ker \Nr$ and $A^\times_1:=A^\times
\cap \ker \Nr$. 
Note that the abelian groups $\Nr(A^\times)=(A^\times)^{\deg(D/K)}$ and
$\Nr(R_i^\times)$ are subgroups of finite index in
$A^\times$. Therefore, the quotient group $\Nr(R_i^\times)/\Nr(A^\times)$
is a finite abelian group.  
As the group $R_{i,1}^\times/A^\times_1$ is finite, one concludes
that $R_i^\times/A^\times$ is also finite.  
Define the mass $\Mass(D,R)$ by
\begin{equation}
  \label{eq:24}
  \Mass(D,R):=\sum_{i=1}^h\, [R_i^\times:A^\times]^{-1}.
\end{equation}
The definition is independent of the choice of 
the representatives $I_i$. 

When $|S|=1$, the group $G(K)=D^\times$ 
has finite $S$-arithmetic subgroups and
hence the mass $\Mass(G,U)$ is also defined, where $U:=\wh R^\times$. 
In this case put
\begin{equation}
  \label{eq:2.5}
  \Mass^{\rm u}(D,R):=\Mass(G,U)=\sum_{i=1}^h\, |R_i^\times|^{-1},
\end{equation}
which is an un-normalized version for $\Mass(D,R)$. 
Clearly we have $\Mass(D,R)=|A^\times|\cdot\Mass^{\rm u}(D,R)$.\\




\section{Comparison of masses}
\label{sec:03}

In the rest of this paper we let $K$, $S$, $A$, $D$ and $R$ be as in
Section~\ref{sec:01} (or \ref{sec:23}), except in Section~\ref{sec:41} 
where $A$ denotes an arbitrary Dedekind domain. 

\subsection{Notation}
\label{sec:31}

Let
$G=D^\times$ be the multiplicative group of $D$, viewed as an
algebraic group over $K$. Let $Z$ be the center of $G$ and $G^{\ad}=G/Z$
be the adjoint group of $G$. We have a short exact sequence of
algebraic groups over $K$:
\begin{equation}
  \label{eq:31}
  \begin{CD}
  1 @>>> Z @>>> G @>{\pr} >> G^\ad @>>> 1
\end{CD},
\end{equation}
where $\pr$ is the natural projection morphism. Let $\Gm$ denote the
multiplicative group over $K$, and $\Nr:G\to \Gm$ be the morphism
induced from the reduced norm map $\Nr:D^\times \to K^\times$. Let
$G_1:=\ker \Nr\subset G$ be the reduced norm one subgroup.  We have a
short exact sequence of 
algebraic groups over $K$:
\begin{equation}
  \label{eq:32}
  \begin{CD}
  1 @>>> G_1 @>>> G @>{\Nr} >> \Gm @>>> 1
\end{CD}.
\end{equation}
The group $G_1$ is an inner form of $\SL_n$ and hence is semi-simple and
simply connected. 

Applying Galois cohomology to (\ref{eq:31}) and using 
Hilbert Theorem 90, we have 
\[ G^\ad(K_v)=G(K_v)/K_v^\times\quad \text{and\ \ }
G^\ad(K)=D^\times/K^\times \]
and that $\pr: G(K_v)\to G^\ad(K_v)$ (resp: $\pr: G(K_v)\to G^\ad(K_v)$) is
a natural surjective map. 
When $v$ is an unramifield place for $D$, we have
$G^\ad(K_v)=\GL_n(K_v)/K_v^\times$. It is not hard to show that 
any maximal open compact subgroup
is conjugate to $\GL_n(O_v)K_v^\times/ K_v^\times$, for example 
using the Cartan decomposition. It follows that
$\pr(G(O_v))$ is a maximal open compact subgroup for almost all places
$v$, and hence that the map $\pr: G(\A^S)\to G^\ad(\A^S)$ is surjective and
open in the adelic topology.  

For any open compact subgroup $U\subset G(\A^S)$, we write $U^\ad$ for
the image $\pr(U)$ of $U$ in $G^\ad(\A^S)$, 
which is an open and compact subgroup. 
Note that
$G^\ad(K_v)=D_v^\times /K^\times_v$ is compact for all $v\in S$ as
$D_v$ is a division algebra and $D^\times_v\simeq \Z\times
O_{D_v}^\times$ (unit group of the unique maximal order). It follows
that the group $G^\ad$ has finite $S$-arithmetic subgroups and 
that $\Mass(G^\ad, U^\ad)$ is defined.



\subsection{Compare $\Mass(D,R)$ and $\Mass(G^\ad, U^\ad)$}
\label{sec:32}
We now take $U=\wh R^{\times}$ and want to compare the mass $\Mass(D,R)$ with
the mass $\Mass(G^\ad, U^\ad)$, where $U^\ad=\pr(U)$.

The projection map $\pr:G(\A^S)\to G^\ad(\A^S)$ gives rise to a surjective map 
$\pr: \DS(G,U)\to \DS(G^\ad, U^\ad)$. Moreover it induces a canonical bijection
 \begin{equation}   \label{eq:33}
   D^\times \backslash G(\A^S)/
   \A^{S,\times} \wh R^\times\simeq \DS(G^\ad, U^\ad). 
 \end{equation} 


Let $\Pic(A)=\A^{S,\times}/K^\times \wh A^\times$ denote the Picard group
of $A$ and let $h_A=|\Pic(A)|$ denote the class number of $A$. 
The group $\Pic(A)$ acts on $\DS(G,U)$ by $[a]\cdot [c]=[ca]$ for $a\in
\A^{S,\times}$ and $c\in G(\A^S)$, where $[a]$ is the class of $a\in
\A^{S,\times}$ in $\Pic(A)$ and $[c]$ is the class in $\DS(G,U)$. 
One has the induced bijection 
\begin{equation}
  \label{eq:335}
  \pr: \DS(G,U)/\Pic(A) \isoto \DS(G^\ad, U^\ad).
\end{equation}
For $c\in G(\A^S)$, write $[c]^\ad$ for the class $D^\times c \A^{S,\times}
\wh R^\times$ and regard it as an element in $\DS(G^\ad,
U^\ad)$ through the canonical isomorphism in (\ref{eq:33}). 

By definition, we have
\[ \Mass(G^\ad,U^\ad)=\sum_{[c]^\ad\in \DS(G^\ad, U^\ad)}
|\Gamma_{c}^\ad|^{-1},  \]
where $\Gamma_{c}^\ad=G^\ad(K)\cap \pr(c)U^\ad \pr(c)^{-1}$.
We have
\begin{equation}
  \label{eq:34}
  \Gamma_{c}^\ad=(D^\times\cap c\wh R^\times c^{-1}
  \A^{S,\times})/K^\times. 
\end{equation}
This group contains $(D^\times\cap c\wh R^\times
c^{-1} K^\times)/K^\times=R_c^\times/A^\times$ as a subgroup, where 
$R_c=D\cap c\wh R c^{-1}$, which is also the left order of the ideal
class corresponding to the class $[c]$. 
Therefore, the contribution of the class
$[c]^\ad$ in $\Mass(G^\ad, U^\ad)$ is equal to
\begin{equation}
  \label{eq:35}
  |\Gamma_{c}^\ad|^{-1}=|R_c^\times/A^\times|^{-1} |(D^\times\cap c\wh
   R^\times c^{-1} \A^{S,\times})/(D^\times \cap K^\times c\wh
   R^\times c^{-1})|^{-1}.  
\end{equation}

On the group $G$, we have
\[ \pr^{-1}([c]^\ad)=\{[ac]; a\in \A^{S,\times}\}\simeq \Pic(A)/\Stab([c]), \]
where $\Stab([c])$ is the stabilizer of the class $[c]$ under the
$\Pic(A)$-action, and 
\[ R_{ac}^\times=\Gamma_{ac}=D^\times \cap (ac)\wh R^\times
(ac)^{-1}=\Gamma_c=R^\times_c.\]
This says that every member in the fiber $\pr^{-1}([c]^\ad)$ has the
same weight. 
Thus,  the weight sum over the fiber
$\pr^{-1}([c]^\ad)$ in $\Mass(D,R)$ is  
\begin{equation}
  \label{eq:36}
  \sum_{[c']\in \pr^{-1}([c]^\ad)} |R_{c'}^\times/A^\times|^{-1}=
 |R_c^\times/A^\times|^{-1} \frac{h_A}{|\Stab([c])|}.
\end{equation}
 It is easy to see 
\[ [ac]=[c]\iff D^\times ac \wh R^\times=D^\times c \wh R^\times \iff
a\in \A^{S,\times}\cap D^\times c\wh R^\times c^{-1}, \]
and we get
\begin{equation}
  \label{eq:37}
  \Stab([c])= (\A^{S,\times}\cap D^\times c\wh R^\times
c^{-1})/K^\times \wh A^\times.
\end{equation}
 
We now show 

\begin{lemma}\label{31}
  There is an isomorphism of finite abelian groups  
  \begin{equation}
    \label{eq:38}
    \Stab([c])\simeq (D^\times\cap \A^{S,\times} c\wh
   R^\times c^{-1} )/(D^\times \cap K^\times c\wh R^\times c^{-1}).
  \end{equation} 
\end{lemma}
\begin{proof}
  To simply notation, put $W:=c\wh R^\times c^{-1}$.
  First of all for $a\in \A^{S,\times}$ we have
\[ a W \cap D^\times\neq \emptyset\iff a\in
   \A^{S,\times}\cap D^\times W. \]
We now show that for each $a\in
   \A^{S,\times}\cap D^\times W$, the intersection $a W \cap D^\times$
   defines an element in $(\A^{S,\times} W\cap D^\times)/(K^\times
   W\cap D^\times)$. Suppose we have two elements $ax_1=d_1$,
   $ax_2=d_2$, where $x_1,x_2\in W$ and $d_1,d_2\in D^\times$. Then 
\[ (ax_1)^{-1} (ax_2)=x^{-1}_1 x_2=d_1^{-1} d_2 \in W\cap D^\times
\subset K^\times W\cap D^\times. \]
Therefore, we define a map
\[ \A^{S,\times}\cap D^\times W \to (\A^{S,\times} W \cap
D^{\times})/(K^\times W \cap D^\times), \quad a \mapsto 
[aW\cap D^\times]. \]
We need to show that elements which go to the identity class lie in
$K^\times \wh A^\times$. Suppose an element $ax\in aW\cap D^\times$ 
lies in the identity class, 
i.e. $ax=k y$ for some $k\in K^\times$ and $y\in W$. Then the element
$ak^{-1}= yx^{-1}$ lies in $\A^{S,\times}\cap W=\wh A^\times$. This
shows that $a\in K^\times \wh A^\times$. Therefore, the above map
induces a bijection
\[ (\A^{S,\times}\cap D^\times W)/K^\times \wh A^\times \simeq 
(\A^{S,\times} W \cap D^{\times})/(K^\times W \cap D^\times). \]
Moreover, this is an isomorphism of finite abelian groups. Combining
with the isomorphism (\ref{eq:37}), 
one obtains an isomorphism (\ref{eq:38}). \qed
\end{proof}

\begin{thm}\label{32}
  We have the
  equality
  \begin{equation}
    \label{eq:39}
    \Mass(D,R)=h_A\cdot \Mass(G^\ad, U^\ad). 
  \end{equation}
\end{thm}

\begin{proof}
It follows from (\ref{eq:35}) and Lemma~\ref{31} that
\[  |\Gamma_c^\ad|^{-1}=|R_c^\times/A^\times|^{-1}
|\Stab([c])|^{-1}. \]
By (\ref{eq:36})
we have 
\begin{equation*}
    \Mass(D,R) =\sum_{[c]^\ad} \ \sum_{[c']\in
  \pr^{-1}([c]^\ad)} |R^\times_{c'}/A^\times|^{-1}  
    =\sum_{[c]^\ad}
  |R^\times_{c}/A^\times|^{-1} \frac{h_A}{|\Stab([c])|}, 
\end{equation*}
where $[c]^\ad$ runs over all double cosets in $\DS(G^\ad,U^\ad)$.
Thus, 
\[ \Mass(D,R) =\sum_{[c]^\ad}
  h_A |\Gamma_c^\ad|^{-1}=h_A \cdot\Mass(G^\ad,U^\ad. \quad \text{\qed} \]
\end{proof}


\begin{cor}\label{33}
  If $R$ and $R'$ are two $A$-orders in $D$, then we have
  \begin{equation}
    \label{eq:311}
    \Mass(D,R)=\Mass(D,R') [\wh R'^\times: \wh R^\times],
  \end{equation}
  where the index $[\wh R'^\times: \wh R^\times]$ is defined in
  {\rm (\ref{eq:23})}. 
\end{cor}
\begin{proof}
  Since both the groups $\wh R'^\times$ and $\wh R^\times$ contain
  the center $\wh A^\times$, one has 
  \[ [\wh R'^\times: \wh R^\times]=[U'^\ad: U^\ad], \] 
  where $U'^\ad=\pr(\wh R'^\times)$ and $U^\ad=\pr(\wh
  R^\times)$. As 
\[ \Mass(G^\ad, U'^\ad)=\Mass(G^\ad,
  U^\ad)[U'^\ad:U^\ad], \]
the assertion follows immediately from
Theorem~\ref{32}.  \qed
\end{proof}

\begin{remark}\label{34} \
  (1) When the class number $h_A$ of $A$ is one, the induced map 
  $\pr:\DS(G,U)\to \DS(G^\ad, U^\ad)$ below
  (\ref{eq:33}) is bijective. In this case the equality
  $\Mass(D,R)=\Mass(G^\ad, U^\ad)$ of different 
  masses in Theorem~\ref{32} 
  is the term-by-term equality.
  


  (2) The action of $\Pic(A)$ on the class space $\DS(G, U)\simeq \Cl(R)$ 
  needs not to be free in general. Therefore, the class number 
  $h(R)=|\DS(G,U)|$ may not be equal to $h_A \cdot |\DS(G^\ad ,U^\ad)|$. 
  To see this,
  let us look at the isotropy subgroup of the identity class $[1]$
  ($c=1$ in (\ref{eq:37})):
\[ \Stab([1])\simeq (\A^{S,\times}\cap D^\times \wh R^\times)/K^\times
  \wh A^\times. \]
In the extreme case one considers the possibility of the equality 
\[ \A^{S,\times} \cap D^\times \wh R^\times=\A^{S,\times}. \]
This is possible if one can find a maximal subfield $L$ of $D$ over
$K$ which satisfies the {\it Principal Ideal Theorem} (cf.
Artin and Tate \cite[Chapter XIII, Section 4,
p.137--141]{artin-tate:cft}),  
that is, $\A^{S,\times}\subset
L^\times \wh B^\times$, where $B$ is the integral closure of $A$ in
$L$. Below is an example (provided by F.-T. Wei). 

(3) {\bf An example.} 
Let $K=\Q(\sqrt{10})$ and $L=K(\sqrt{-5})=\Q(\sqrt{-5},\sqrt{-2})$. 
Let $D$ be the quaternion algebra over $K$ which is ramified exactly at
the two real places of $K$. Since $L/K$ is inert at the real
places, we can embed $L$ into $D$ over $K$ by the Hasse principle
(cf. \cite[Section 18.4]{pierce}). 
Notice that the primes $2$ and $5$
are ramified in $K$. Let $\grp$ be the prime of $O_K=\Z[\sqrt{10}]$
lying over $5$. 

Claim: $\grp=\sqrt{10}\,O_K+5 O_K$ and $\grp$ is of order $2$ in
$\Pic(O_K)$. 

Proof of the claim: Let $\grq$ be the unique prime of $O_K$ lying over
$2$. Then $\sqrt{10}\, O_K=\grp \grq$ and $5O_K=\grp^2$. Therefore,
$\grp=\sqrt{10}\,O_K+5 O_K$, and $\grp^2=5O_K$ is principal. 
We now show that $\grp$ is not principal. Suppose that $\grp$ is
principal. Then there exist $x,y\in \Z$ such that $\Nr(x+y
\sqrt{10})=x^2-10 y^2=\pm 5$. Then $x=5x'$ for some $x'\in \Z$, and
$5x'^2-2y^2=\pm 1 \equiv \pm 1 \pmod 5$. This implies that
$-2y^2\equiv \pm 1 \pmod 5$, which is a contradiction. 
 
Moreover, we have 
\[ \grp O_L=\sqrt{10}\,O_L+5 O_L=\sqrt{-5}\, (\sqrt{-2}\,
O_L+\sqrt{-5}\, O_L)=\sqrt{-5}\, O_L, \] 
which is principal. Let $R$ be a
maximal order in $D$ which contains $O_L$. Then $\grp R=\sqrt {-5} R$.  
This shows that the isotropy subgroup of the identity class $[1]$ is
non-trivial, and particularly that 
the action of $\Pic(O_K)$ on $\Cl(R)$ is not free. 
As the class number of $O_K$ is equal to $2$,
we also show that the canonical map $\Pic(O_K)\to \Pic(O_L)$, 
sending any ideal class $[I]$ to $[I O_L]$, is the zero map. 

\end{remark}



\subsection{Comparison 
of $\Mass(G^\ad,U^\ad)$ and $\Mass(G_1,U_1)$} 
\label{sec:34}
Recall that $G_1$ is the norm-one subgroup of $G$ and $U_1:=U\cap
G_1(\A^S)$, where $U=\wh R^\times$. Let $\wt R$ be a maximal $A$-order
in $D$ containing $R$. Put $\wt U:=(\wt R\otimes_A \wh A)^\times$ and
$\wt U_1:=\wt U\cap G_1(\A^S)$. We compare the masses
$\Mass(G^\ad,U^\ad)$ and $\Mass(G_1,U_1)$. 
Using the interpretation of masses as
the volume of fundamental domains (Lemma~\ref{22}), one first has 
\begin{equation}
  \label{eq:3.13}
  \begin{split}
  \Mass(G_1,U_1) & = \Mass(G_1, \wt U_1)[\wt U_1:U_1],   \\
  \Mass(G^\ad,U^\ad) & =\Mass(G^\ad, \wt U^\ad)[\wt U^\ad:U^\ad].
  \end{split}
\end{equation}
From this we see that the comparison of these two masses depends on
$U$ and can be reduced to the case where $R$ is a maximal $A$-order. Put
\begin{equation}
  \label{eq:3.14}
  c(S,U):=\frac{\Mass(G^\ad,U^\ad)}{\Mass(G_1,U_1)}.
\end{equation}
\begin{lemma} \label{35}
One has
\begin{equation}
  \label{eq:3.15}
  c(S,U)=c(S,\wt U) \cdot [\wh A^\times: \Nr(U)],
\end{equation}
where $\Nr:G(\A^S)\to \A^{S,\times}$ is the reduced norm map. 


\end{lemma}
\begin{proof}
  Using the relation (\ref{eq:3.13}) we get 
  \begin{equation}
    \label{eq:3.16}
     c(S,U)=c(S,\wt U)\cdot \frac{[\wt U^\ad: U^\ad]}{[\wt U_1:U_1]}.
  \end{equation}
Since both $U$ and $\wt U$ contain the center $\wh A^\times$, one
has $[\wt U^\ad: U^\ad]=[\wt U: U]$. 
Using the following short exact sequences
\[ 
\begin{CD}
  1 @>>> U_1 @>>> U @>>> \Nr(U) @>>> 1,  
\end{CD} \]
\[ \begin{CD}
  1 @>>> \wt U_1 @>>> \wt U @>>> \Nr(\wt U)=\wh A^\times @>>> 1,  
\end{CD} 
\]
one easily shows that $[\wt U:U]=[\wt U_1: U_1] \cdot [\wh A^\times:
\Nr(U)]$. Thus, $[\wt U^\ad: U^\ad]=[\wt U_1: U_1] \cdot [\wh A^\times:
\Nr(U)]$ and the lemma is proved. \qed
\end{proof}


Recall (Section~\ref{sec:23}) that $S_D\subset V^K$ denotes the finite set of 
ramified places for $D$.

\begin{thm} Assume that $S=\infty$ and
  that $R$ is a maximal $A$-order.  
\begin{enumerate}
\item If $K$ is a totally real number field, then  
\begin{equation}
    \label{eq:3.17}
    \Mass(D, R)=h_A\cdot \frac{(-1)^{[K:\Q]}}{2^{[K:\Q]-1}}\cdot
\zeta_K(-1)\cdot \prod_{v\in S_D\cap V^K_f} (q_v-1),
\end{equation}
\begin{equation}
    \label{eq:3.18}
    \Mass(G^\ad, U^\ad)=\frac{(-1)^{[K:\Q]}}{2^{[K:\Q]-1}}\cdot
     \zeta_K(-1)\cdot \prod_{v\in S_D\cap V^K_f} (q_v-1),
\end{equation}
and 
\begin{equation}
  \label{eq:3.19}
  \Mass(G_1, U_1)=\frac{(-1)^{[K:\Q]}}{2^{[K:\Q]}}\cdot 
\zeta_K(-1)\cdot \prod_{v\in S_D\cap V^K_f} (q_v-1).
\end{equation}
\item If $K$ is a global function field, then
\begin{equation}
    \label{eq:3.20}
    \Mass(D, R)=h_A 
    \cdot 
\prod_{i=1}^{n-1} \zeta_K(-i)\cdot \prod_{v\in S_D} \lambda_v,
\end{equation}
\begin{equation}
    \label{eq:3.21}
   \Mass(G^\ad, U^\ad)=
\prod_{i=1}^{n-1} \zeta_K(-i)\cdot \prod_{v\in S_D} \lambda_v, 
  \end{equation}
and 
\begin{equation}
  \label{eq:3.22}
  \Mass(G_1, U_1)= 
\prod_{i=1}^{n-1} \zeta_K(-i)\cdot \prod_{v\in S_D} \lambda_v.
\end{equation}
where 
\begin{equation}
  \label{eq:3.23}
  \lambda_v=\prod_{1\le i\le n-1, d_v\nmid i} (q_v^i-1)
\end{equation}
and $d_v$ is the index of $D_v:=D\otimes_K K_v$. 
\end{enumerate}
\end{thm}
\begin{proof}
  (1) The formulas for $\Mass(D,R)$ and $\Mass(G_1, U_1)$ are 
  due to Eichler \cite{eichler:crelle55}; also see \cite[Chapter
  V]{vigneras}.  
  The formula for $\Mass(G^\ad, U^\ad)$ 
      follows from Eichler's formula for $\Mass(D,R)$ and
      Theorem~\ref{32}.

  (2) The formula for $\Mass(D,R)$ is obtained by 
      Denert and Van Geel \cite{denert-geel:classno88} and also by Wei
      and the author \cite[Theorem 1.1]{wei-yu:mass_division}.  
      The formula for $\Mass(G_1,U_1)$ follows from the relation
      $\Mass(D,R)=h_A \cdot \Mass(G_1,U_1)$; see
      \cite[Eq. (3), p.~907]{yu-yu:ssd}
\footnote{In the function field case with $|S|=1$ the notation $\Mass(D,R)$ in
\cite{wei-yu:mass_division} is defined to be 
the un-normalized mass $\Mass^{\rm u}(D,R)$ (\ref{eq:2.5}) in
this paper, which is $(q-1)^{-1}$ times the mass 
$\Mass(D,R)$ in this paper.}. 
      The formula for $\Mass(G^\ad, U^\ad)$ follows from the formula
      for $\Mass(D,R)$ and Theorem~\ref{32}. \qed
\end{proof}

\begin{thm}\label{37}
  Assume that $R$ is a maximal
  $A$-order. We have 
  \begin{equation}
    \label{eq:3.24}
    \begin{split}
    \Mass(G^\ad, U^\ad)& =c^\ad \cdot
    \prod_{i=1}^{n-1}|\zeta_K(-i)| \cdot \prod_{v\in S_D\cap V^K_f} \lambda_v, \\
    \Mass(G_1, U_1) & =c_1\cdot \prod_{i=1}^{n-1}|\zeta_K(-i)| \cdot
    \prod_{v\in S_D\cap V^K_f} \lambda_v, 
    \end{split}
  \end{equation}
where $\lambda_v$ is given in (\ref{eq:3.23}), 
$c^\ad=1/n^{|S|-1}$ 
and 
\begin{equation}
  \label{eq:3.25}
c_1=
\begin{cases}
   1 & \text{if $K$ is a function field;} \\
   {2^{-[K:\Q]}}  & \text{if $K$ is a totally real number field.} 
\end{cases}  
\end{equation}
\end{thm}
\begin{proof}
We write $\Mass(G^\ad, U^\ad, S)$ for $\Mass(G^\ad, U^\ad)$ to
emphasize the dependence of the mass on $S$. We have
\begin{equation}\label{eq:3.26}
\begin{split}
  \Mass(G^\ad,U^\ad,S)& =\frac{\vol(G^\ad(K)\backslash
  G^\ad(\A^S))}{\vol(U)} \\
  & = \frac{\vol(G^\ad(K)\backslash
  G^\ad(\A^\infty))}{\vol(\prod_{v\in S-\infty} G^\ad(O_v)\cdot U)}
  \cdot \prod_{v\in S-\infty} \left [ \frac{\vol(G^\ad(K_v))}{\vol(G^\ad
  (O_v))} \right ]^{-1} \\
  & = \frac{1}{n^{|S|-|\infty|}} \frac{\vol(G^\ad(K)\backslash
  G^\ad(\A^\infty))}{\vol(\prod_{v\in S-\infty} G^\ad(O_v)\cdot U)}\\
  & = \frac{1}{n^{|S|-|\infty|}} \cdot \Mass(G^\ad,U^\ad, \infty).
\end{split}  
\end{equation}
Here $G^\ad(O_v)=O_{D_v}^\times/O_{v}^\times$ where $O_{D_v}$ is the
valuation ring in the division algebra $D_v$, and we use
the isomorphism $G^\ad(K_v)/G^\ad(O_v)\simeq \Z/n\Z$. 
The computation above reduces to the case where $S=\infty$. 
Using the formulas (\ref{eq:3.18}) and (\ref{eq:3.21}) we compute the
factor 
\[ c^\ad=\frac{1}{n^{|S|-|\infty|}} \cdot
\frac{1}{n^{|\infty|-1}}=\frac{1}{n^{|S|-1}}. \]
This settles the formula for $\Mass(G^\ad, U^\ad)$. 
Using $G_1(K_v)=G_1(O_v)$ for $v\in S$,
the same computation as in (\ref{eq:3.26}) shows that $\Mass(G_1,U_1,
S)=\Mass(G_1, U_1,\infty)$, i.e.
$\Mass(G_1,U_1)$ is independent of $S$. Therefore, the formula for
$\Mass(G_1,U_1)$ is given by (\ref{eq:3.19}) and
(\ref{eq:3.22}), respectively. \qed
\end{proof}



We now show the following comparison result. 

\begin{cor}\label{38}
  Let $R$ be any $A$-order in $D$. 
  We have
  \begin{equation}
    \label{eq:3.27}
    \Mass(G^\ad, U^\ad)= c(S,U) \cdot \Mass(G_1,U_1), 
  \end{equation}
where 
\begin{equation}
  \label{eq:3.28}
  c(S,U)=
  \begin{cases}
   n^{-(|S|-1)} [\wh A^\times: \Nr(U)]  & \text{if $K$ is a
   function field;} \\ 
   2^{-(|S|-|\infty|-1)} [\wh A^\times: \Nr(U)]  & \text{if $K$ is a
   totally real number field.} \\  
  \end{cases}
\end{equation}
\end{cor}
\begin{proof}
  When $R$ is a maximal order, we compute using Theorem~\ref{37}
  \begin{equation}
    \label{eq:3.29}
   c(S,\wt U)=\begin{cases}
   n^{-(|S|-1)}   & \text{if $K$ is a
   function field;} \\ 
   2^{-(|S|-|\infty|-1)}  & \text{if $K$ is a
   totally real number field.}\\  
  \end{cases}
  \end{equation}
The statement then follows from Lemma~\ref{35}. \qed
\end{proof}

The proof of Corollary~\ref{38} when $R$ is a maximal $A$-order is ad
hoc. Namely, this is derived after knowing both $\Mass(G^\ad, U^\ad)$
and $\Mass(G_1,U_1)$.




\section{Mass formulas for arbitrary $A$-orders $R$}
\label{sec:04}

In the previous section 
we obtain the formulas for $\Mass(D,R)$,
$\Mass(G^\ad, 
U^\ad)$ and $\Mass(G_1,U_1)$ in the case where 
$R$ is maximal. 
We now consider the case of arbitrary $A$-orders $R$. 
Using Theorem~\ref{32} and
Corollary~\ref{38}, one only needs to know any of them.
We derive a formula for $\Mass(D,R)$.

\subsection{More notations}
\label{sec:41}
Let $A$ be any Dedekind domain and let $K$ be the fraction field of
$R$. Let $V$ be a finite-dimensional $K$-vector space. 
For any two (full) $A$-lattices $X_1$ and $X_2$, 
let $\chi(X_1,X_2)$ be the unique fractional ideal of
$A$ that is characterized by the following properties 
(See Serre \cite[Chapter III, Section 1]{serre:lf}):
\begin{itemize}
\item If $X_2\subset X_1$ and $X_1/X_2\simeq A/\grp$ for a non-zero
  prime ideal 
  $\grp\subset A$, then $\chi(X_1,X_2)=\grp$. 
\item $\chi(X_1, X_2)=\chi(X_2,X_1)^{-1}$ for any two $A$-lattices
  $X_1$ and $X_2$ in $V$.
\item $\chi(X_1,X_2) \chi(X_2,X_3)=\chi(X_1,X_3)$ for any three
  $A$-lattices $X_1, X_2$ and $X_3$ in $V$. 
\end{itemize}

When $K$ is a global field, 
we define $|I|$ to be $|A/I|$ for any
non-zero integral ideal $I\subset A$ and extend the definition to
fractional ideals by \[ |I_1 I_2^{-1}|=|I_1| |I_2|^{-1}\]
for non-zero integral
ideals $I_1$ and $I_2$ of $A$. In this case let $\wh A$ denote the finite
completion of $A$ and $\wh K:=\wh A\otimes_A K$. Put $\wh X:=X\otimes_A
\wh A$ and $\wh V:=V\otimes_K \wh K$. 
Then for any Haar measure on $\wh V$ one has
\begin{equation}
  \label{eq:4.1}
  |\chi(X_1,X_2)|=\frac{\vol(\wh X_1)}{\vol(\wh X_2)}. 
\end{equation}

Now we define the {\it discriminant} of an $A$-lattice with
respect to a bilinear form on $V$ (for any Dedekind domain $A$). 
Let $T:V\times V\to
K$ be a non-degenerate $K$-bilinear map. 
Put $n=\dim_K V$. For any $K$-basis
$E=\{e_1,e_2, \dots, e_n\}$ of $V$, the {\it discriminant of $E$} with
respect to $T$ is defined to be 
\begin{equation}
  \label{eq:4.2}
  D_T(E):=\det(T(e_i,e_j))\in K. 
\end{equation}
For an $A$-lattice $X$ in $V$, the {\it discriminant of $X$} with
respect to $T$ is defined to be the fractional ideal generated by $D_T(E)$
\begin{equation}
  \label{eq:4.3}
  \grd_{T}(X):=(D_T(E))_E\subset K
\end{equation}
for all $K$-bases $E$ contained in $X$.
Computation of discriminants 
can be reduced to the local computation, namely, we have
\begin{equation}
  \label{eq:4.4}
  \grd_T(X)\otimes_A A_\grp =\grd_T(X_\grp), \quad
  X_\grp:=X\otimes_A A_\grp, 
\end{equation} 
where $A_\grp$ is the completion of $A$ at the non-zero prime ideal
$\grp$.  

If $X_1$ and $X_2$ are two $A$-lattices in $V$, then one has the formula
\cite[Chap.~III, \S~2, Proposition 5, p.~49]{serre:lf}
\begin{equation}
  \label{eq:4.5}
  \grd_T(X_2)=\grd_T(X_1) \chi(X_1,X_2)^2. 
\end{equation}
In particular, if $X_2\subset X_1$ then $\grd_T(X_2)=\grd_T(X_1)
\gra^2$, where $\gra=\chi(X_1,X_2)$, which is an integral ideal of $A$.

Now we define the {\it reduced discriminant} of an $A$-lattice in a
{\it central simple algebra} over $K$; some authors simply call this the
{\it discriminant} of the lattice. Let $B$ be a central 
simple $K$-algebra and $X$ be an $A$-lattice in $B$. Let $T:B\times
B\to K$ be the non-degenerate $K$-bilinear form defined by 
\[ T(x,y):=\Tr(x\cdot y), \]
where $\Tr:B\to K$ is the reduced trace from $B$ to $K$. Then 
$\grd_T(X)$ is defined and it can be shown to be the square of a
unique fractional ideal $\gra$ in $K$. 
The {\it reduced discriminant of $X$}, 
denoted by $\grd(X)$, is defined to this fractional ideal $\gra$,
namely, the square root of $\grd_T(X)$. It is easy to see that the
association $X \mapsto \grd(X)$ commutes with finite etale base changes and 
localizations. Namely, if $A'$ is a
finite etale extension or a localization of $A$ then one has
\begin{equation}
  \label{eq:4.6}
  \grd(X\otimes_A A')=\grd(X)\otimes_A A'. 
\end{equation}

\subsection{Computation of $\Mass(D,R)$}
\label{sec:42}
We return to compute $\Mass(D, R)$ where $R$ is any $A$-order.
Let $\wt R$ be a maximal $A$-order in $D$ containing $R$.    
The masses $\Mass(D,\wt R)$ and $\Mass(D,R)$ differ by the
factor
\begin{equation}
  \label{eq:4.7}
  \prod_{v\not \in S} [\wt R^\times_v: R_v^\times],
\end{equation}

Put $\kappa(R_v):=R_v/\rad(R_v)$, where $\rad(R_v)$ denotes the
Jacobson radical of $R_v$. 

\begin{lemma}\label{41} \
  
\begin{enumerate}
\item We have
  \begin{equation}
    \label{eq:4.8}
 [\wt R_v:R_v]=\frac{|\grd(R_v)|}{|\grd(\wt R_v)|}.   
  \end{equation}
\item We have
  \begin{equation}\label{eq:4.9}
  [\wt R_v^\times:R_v^\times]=\frac{|\grd(R_v)|}{|\grd(\wt
    R_v)|} \cdot \frac{|\kappa(\wt R_v)^\times|/|\kappa(\wt
    R_v)|}{|\kappa(R_v)^\times|/|\kappa(R_v)|}.  
  \end{equation}
\end{enumerate}
\end{lemma}
\begin{proof}
  (1) We have $[\wt R_v:R_v]=|\chi(\wt R_v, R_v)|$\ from
      (\ref{eq:4.1}) and
      $\grd(R_v)=\grd(\wt R_v)\chi(\wt R_v, R_v)$. Then we get
      $|\grd(R)|=|\grd(\wt R_v)|\cdot [\wt R_v: R_v]$ and
      (\ref{eq:4.8}). 

  (2) For any Haar measure on $D_v$ we have
\[ [\wt R_v^\times: R_v^\times]=
  \frac{\vol(\wt R_v^\times)}{\vol(R_v^\times)}=\frac{\vol(\wt
  R_v)}{\vol(R_v)} \cdot \frac{|\kappa(\wt R_v)^\times|/|\kappa(\wt
    R_v)|}{|\kappa(R_v)^\times|/|\kappa(R_v)|} \] 
Then we obtain the formula (\ref{eq:4.9}) from the formula
  (\ref{eq:4.8}). \qed
\end{proof}

For any non-Archimedean place $v\in S$, we define
$R_v$ to be the unique maximal order $O_{D_v}$ in the
division algebra $D_v$, noting that this is not the completion of $R$,
which does not make sense. 
For any non-Archimedean place $v$, we define
\begin{equation}
  \label{eq:4.10}
  \lambda_v(R_v):=\frac{|\grd(R_v)|}{|\kappa(R_v)^\times|/|\kappa(R
  _v)|}\cdot \prod_{1\le i\le n} (1-q_v^{-i}). 
\end{equation}
Clearly $\lambda_v(R_v)=1$ when $R_v\simeq \Mat_n(A_v)$. 
Now we prove the following formula.

\begin{thm}\label{42} Notations as above. We have 
  \begin{equation}
    \label{eq:4.11}
    \Mass(D,R)=h_A \cdot \frac{1}{n^{|S|-1}} \cdot
    \prod_{i=1}^{n-1}|\zeta_K(-i)| \cdot \prod_{v\in V^K_f} \lambda_v(R_v).
  \end{equation}
\end{thm}
\begin{proof}
  By Theorem~\ref{32}, Corollary~\ref{33}  and Theorem~\ref{37} we have
  \begin{equation}
    \label{eq:4.12}
    \Mass(D,R)=h_A\cdot \frac{1}{n^{|S|-1}} \cdot
    \prod_{i=1}^{n-1} | \zeta_K(-i) |  \cdot \prod_{v} 
    (\lambda_v \cdot [\wt R_v^\times :R_v^\times]),
  \end{equation}
  where $\lambda_v$ is defined in (\ref{eq:3.23}). 
  Thus, it suffices to check 
  \begin{equation}
    \label{eq:4.13}
    \lambda_v \cdot [\wt R_v^\times :R_v^\times]=\lambda_v(R_v).
  \end{equation}
The left hand side of (\ref{eq:4.13}) is equal to (using Lemma~\ref{41})
\begin{equation}
  \label{eq:4.14}
  \lambda_v \cdot \frac{|\grd(R_v)|}{|\grd(\wt R_v)|} 
    \cdot \frac{|\kappa(\wt R_v)^\times|/|\kappa(\wt
    R_v)|}{|\kappa(R_v)^\times|/|\kappa(R_v)|}. 
\end{equation}
Suppose $D_v=\Mat_{m}(\Delta)$, where $\Delta$ is a central division
algebra with index $d$, thus $n=dm$. 
Note that
\begin{equation}
  \label{eq:4.15}
  \begin{split}
  &\lambda_v \cdot \frac{1}{|\grd(\wt R_v)|}
  \cdot |\kappa(\wt R_v)^\times|/|\kappa(\wt R_v)| \\
  = &\prod_{1\le i\le n-1, d\nmid i} (q_v^i-1)\cdot \frac{1}{q_v^{m^2
    \cdot d(d-1)/2}} \cdot \prod_{1\le j\le m}
    (1-q_v^{-dj}) \\ 
  = &\prod_{1\le i\le n} (1-q_v^{-i}).\\ 
\end{split}
\end{equation}
This verifies the equality (\ref{eq:4.13}) 
and completes the proof of the theorem. \qed
\end{proof}

      
In the rest of this section we restrict to the case $n=2$. If the 
order $R_v$ is
not isomorphic to $\Mat_2(A_v)$, then define 
the {\it Eichler symbol} $e(R_v)$ by
\begin{equation}
  \label{eq:4.16}
  e(R_v)=
  \begin{cases}
    1  &  \text{if $\kappa(R_v)=\kappa (v)\times \kappa(v)$;}\\
    -1 & \text{if $\kappa(R_v)$ is a quadratic field extension of
       $\kappa(v)$;}\\
   0  & \text{if $\kappa(R_v)=\kappa(v)$.}\\ 
  \end{cases}
\end{equation}


\begin{cor}\label{43}
  Assume that $n=2$. Then we have
  \begin{equation}
    \label{eq:4.17}
    \Mass(D,R)=\frac{h_A |\zeta_K(-1)|}{2^{|S|-1}}\prod_{v\in S_R}
    |\grd(R_v)| \frac{(1-q_v^{-2})}{(1-e(R_v) q_v^{-1})}. 
  \end{equation}
  where $S_R$ consists of all non-Archimedean places $v$ of $K$ such
  that either $v$ is ramified in $D$ or $R_v$ is not maximal. 
\end{cor}
\begin{proof}
  By Theorem~\ref{42}, it suffices to check 
  \begin{equation}
    \label{eq:4.18}
    \frac{|\kappa(R_v)^\times|}{|\kappa(R_v)|}=(1-q_v^{-1})(1-e(R_v)
    q_v^{-1}).
  \end{equation}
But this is clear. \qed
\end{proof}

In the case where $K$ is a totally real number field and $S=\infty$ 
Corollary~\ref{43} was obtained by K\"orner \cite[Theorem 1]{korner:1987}. 


\section{Mass formulas for types of orders}
\label{sec:05}

Let $\scrR$
be the {\it genus} of $R$, that is, the set consists of all $A$-orders in
$D$ which are isomorphic to $R$ locally everywhere. A {\it type} of
$R$ is a $D^\times$-conjugacy class of orders 
in $\scrR$. The
set of $D^\times$-conjugacy classes of orders in $\scrR$ is denoted by
$T(R)$. This is a 
finite set and its cardinality $|T(R)|$, denoted by $t(R)$, is called
the {\it type number} of $R$. 

\begin{defn}
  Let $\{R_1, \dots, R_t\}$ be a set of $A$-orders representing the
  $D^\times$-conjugacy classes in $\calR$. Define the {\it mass of the
  types of $R$} by
\begin{equation}
  \label{eq:5.1}
  \Mass(T(R)):=\sum_{i=1}^t\,  [N(R_i):K^\times]^{-1},
\end{equation}
where $N(R_i)$ is the normalizer of $R_i$ in $D^\times$. 
\end{defn}

We know that there is a natural bijection
\begin{equation}
  \label{eq:5.2}
  T(R)\simeq D^\times \backslash G(\A^S)/\calN(\wh R), 
\end{equation}
where $\calN(\wh R)$ is the normalizer of $\wh R$ in 
$\wh D^\times=G(\A^S)$.  

The following result evaluates $\Mass(T(R))$. In the computation, one
also shows that each term $[N(R_i):K^\times]$ is finite so that
$\Mass(T(R))$ is defined. 

\begin{thm}\label{52}
  We have
  \begin{equation}
    \label{eq:5.3}
    \Mass(T(R))=\frac{1}{n^{|S|-1}} \cdot
    \prod_{i=1}^{n-1}|\zeta_K(-i)|\cdot 
    \prod_{v} \lambda_v(R_v)  
    \cdot [\calN(\wh R): \A^{S,\times} \wh R^\times]. 
  \end{equation}
\end{thm}
\begin{proof}
  Let $\calN^\ad$ denote the image of the open subgroup $\calN(\wh
  R)\subset G(\A^S)$ in $G^\ad(\A^S)$. We now show 
  \begin{equation}
    \label{eq:5.4}
    \Mass(T(R))=\Mass(G^\ad, \calN^\ad).
  \end{equation}
  Let $c_1, \dots, c_t\in G(\A^S)$ be representatives for the double
  coset space in (\ref{eq:5.2}). For each $i=1, \dots, t$, put 
  \begin{equation}
    \label{eq:5.5}
    \Gamma_i^\ad:=G^\ad(K)\cap \pr(c_i)\, \calN^\ad\, \pr(c_i)^{-1},
    \quad\text{and}\quad  \Gamma_i:=G(K)\cap c_i\, \calN(\wh R)\,
    c_i^{-1}. 
  \end{equation}
  It is clear that $\Gamma_i^\ad=\Gamma_i/K^\times$. So it suffices to
  show that $\Gamma_i=N(R_i)$. Notice $R_i=D\cap c_i \wh R c_i^{-1}$,
  so $\wh R_i=c_i \wh R c_i^{-1}$. Let $x\in \Gamma_i$. Then $x=c_i y
  c_i^{-1}$ for some $y\in \calN(\wh R)$. Therefore, $c_i^{-1} x
  c_i\in \calN(\wh R)$. This gives $x(c_i \wh R c_i^{-1}) x^{-1}=(c_i
  \wh R c_i^{-1})$. Therefore,
  \[ x\in \Gamma_i \iff x(\wh R_i) x^{-1} =\wh R_i, \]
and hence $\Gamma_i=N(R_i)$. This shows (\ref{eq:5.4}).

Using (\ref{eq:5.4}), we have $\Mass(T(R))=\Mass(G^\ad, U^\ad)\cdot
[\calN^\ad: U^\ad]$. Then formula (\ref{eq:5.3}) follows from
Theorems~\ref{42} and~\ref{32} and 
$[\calN^\ad:U^\ad]=[\calN(\wh R): \A^{S,\times} \wh R^\times]$. \qed 
\end{proof}

\section*{Acknowledgments}

The author thanks F.-T. Wei and J.-K. Yu for helpful discussions. The
paper is reorganized while the author's stay in the Max-Planck-Institut
f\"ur Mathematik. He is grateful to the Institut for kind hospitality
and excellent working environment. 
The author was partially supported by the grants 
MoST 100-2628-M-001-006-MY4 and 103-2918-I-001-009.


\end{document}